\documentclass[11pt,twoside]{article}

\usepackage{mathrsfs,amsfonts,amsmath,amssymb}
\usepackage{latexsym}
\newtheorem{satz}{Theorem}
\newtheorem{proposition}[satz]{Proposition}
\newtheorem{theorem}[satz]{Theorem}
\newtheorem{lemma}[satz]{Lemma}
\newtheorem{definition}[satz]{Definition}
\newtheorem{corollary}[satz]{Corollary}
\newtheorem{remark}[satz]{Remark}

\def\no{\noindent}

\def\eps{\varepsilon}
\def\_phi{\varphi}

\def\a{\alpha}
\def\d{\delta}

\def\F{{\mathbb F}}

\def\o{\omega}
\def\ov{\overline}

\def\C{{\mathbb C}}

\def\E{\mathsf {E}}
\def\T{{\mathbb T}}

\def\Z_N{{\mathbb Z}_N}
\def\Z{{\mathbb Z}}
\def\N{{\mathbb N}}

\def\Gr{{\mathbf G}}

\def\D{{\mathbb D}}

\def\l{\left}
\def\r{\right}

\def\Spec{{\rm Spec\,}}

\def\oM{{\rm M}}

\def\oT{{\rm T}}

\def\G{\Gamma}
\def\FF{\widehat}
\def\c{\circ}
\def\D{\Delta}
\def\Cf{{\mathcal C}}

\def\T{\mathsf {T}}

\author{Shkredov I.D.}
\title{ On exponential sums over multiplicative subgroups of medium  size
\footnote{
This work was supported by grant RFFI NN
11-01-00759, Russian Government project 11.G34.31.0053,
Federal Program "Scientific and scientific--pedagogical staff of innovative Russia" 2009--2013,
grant mol\underline{ }a\underline{ }ved 12--01--33080
and
grant Leading Scientific Schools N 2519.2012.1.}
}
\date{}
\begin{document}
\maketitle

\begin{center}
 Annotation.
\end{center}

{\it \small
    In the paper we obtain some new upper bounds for exponential sums over multiplicative
    subgroups $\G \subseteq \F^*_p$ having sizes in the range $[p^{c_1}, p^{c_2}]$,
    where $c_1,c_2$ are some absolute constants close to $1/2$.
    As an application we prove that in symmetric case $\G$  is always an additive basis of order five,
    provided by $|\G| \gg p^{1/2} \log^{1/3} p$.
    Also the method allows us to give a new upper bound for Heilbronn's exponential sum.
}
\\

\section{Introduction}
\label{sec:introduction}

Let $p$ be a prime number, $\F_p$ be the finite field, and $\F^*_p = \F_p \setminus \{0\}$.
Let $\G \subseteq \F^*_p$ be a multiplicative subgroup.
Such subgroups were studied by various authors
(see e.g. \cite{Bou_prod1,Bourgain_new_sum-prod,BGK,BK,Waring_Z_p,Hart_A+A_subgroups,H-K,K_Tula,KS1,s_ineq,Shp-DH}).
One of the main questions in the field is to give a good upper bound for the exponential sums
over multiplicative subgroups.
More precisely, denote by $\oM(\G)$ the maximal nonzero Fourier coefficient over $\G$, that is
$$
    \oM(\G) := \max_{\xi \neq 0} \left| \sum_{x\in \G} e^{\frac{2 \pi i x \xi}{p}} \right| \,.
$$
So, what can we say nontrivial about the quantity $\oM(\G)$?
The question was studied both analytical (e.g. \cite{KS1}) and combinatorial tools (e.g. \cite{BGK}).
One of the main results of the paper is the following.

\begin{theorem}
    Let $\G \subseteq \F_p$ be a multiplicative subgroup, $|\G| \le p^{2/3}$.
    Then
\begin{equation}\label{f:M_2/3_intr}
    \oM (\G) \ll |\G|^{1/2} p^{1/6} \log^{1/6} |\G| \,.
\end{equation}
\end{theorem}

Actually, we obtain a new estimate for the moments of such exponential sums, see inequality (\ref{f:13/2})
as well as a more general bound for sums over arbitrary multiplicative character
of subgroup (\ref{f:char_6}), (\ref{f:char_6'}).
Our estimate  (\ref{f:M_2/3_intr}) is better than the previous bounds
in the range $|\G| \in (p^{52/141}, p^{29/48})$, roughly,
and this explains the name of the paper.

There is a well--known  conjecture that the sumset $2\G := \G + \G$ contains $\F^*_p$, provided by $|\G| > p^{1/2+\eps}$,
where $\eps>0$ is any number.
We consider the question from an inverse perspective.
Let $|\G| > p^{1/2+\eps}$; what is the smallest $k$ such that $k\G$ contains $\F_p^*$?
A.A. Glibichuk proved in \cite{Glibichuk_zam} that $k$ can be taken equals $8$.
After that several authors (see \cite{Hart_A+A_subgroups,ss,ss_E_k,s_ineq,sv})
prove that $k=6$ is possible and, actually, the condition $|\G| > p^{1/2+\eps}$ can be relaxed slightly.
The subgroups of cardinality near $\sqrt{p}$ are exactly what we called "medium"\, size, so
we can say something nontrivial on the question.
That is our second result.

\begin{theorem}
    Let $\Gamma \subseteq \F_p^*$ be a multiplicative subgroup such that $|\Gamma| \gg p^{1/2} \log^{1/3} p$,
    and $-1 \in \G$.
    Then for all sufficiently large $p$, we have $\F^*_p \subseteq 5\Gamma$.
\label{t:6R_subgroups_intr}
\end{theorem}

Another application of our method
is  a new upper bound for so--called Heilbronn's exponential sum (see section \ref{sec:application}),
which is connected \cite{BFKS}, \cite{Chang_Fermat}, \cite{Lenstra}, \cite{OstShp}, \cite{Shp-FermVal}, \cite{Shp-Ihara}
with the distribution of so--called {\it Fermat quotients} defined as
$$
    q(n) = \frac{n^{p-1}-1}{p} \,, \quad n\neq 0 \pmod p \,.
$$
Heilbronn's exponential sum can be expressed as quantity  $\oM (\G)$ of some subgroup of $\Z/p^2\Z$,
$|\G| = p-1$.
So, that is exactly subgroup of "medium"\, (square root of the cardinality of group)
size and our approach can be applied in the case.

Let us say a few words about the method of the proof.
In papers
\cite{s_ineq}---\cite{s_mixed}, \cite{sv_heilbronn}
we obtain a new upper bound for so--called the additive energy
(see the definition in book \cite{tv} or in section \ref{sec:definitions})
of a multiplicative subgroup and after that derive form it an upper bound for the exponential sum over
the
subgroup.
In the proof we calculated the spectrum of some operators.
Here we use a more direct approach counting the quantity $\oM (\G)$ via
operators having another special weights (more precisely, see section \ref{sec:proof}).
The results say nothing new about the additive energy of multiplicative subgroups.


\section{Definitions}
\label{sec:definitions}

Let $\Gr$ be an abelian group.
If $\Gr$ is finite then denote by $N$ the cardinality of $\Gr$.
It is well--known~\cite{Rudin_book} that the dual group $\FF{\Gr}$ is isomorphic to $\Gr$ in the case.
Let $f$ be a function from $\Gr$ to $\mathbb{C}.$  We denote the Fourier transform of $f$ by~$\FF{f},$
\begin{equation}\label{F:Fourier}
  \FF{f}(\xi) =  \sum_{x \in \Gr} f(x) e( -\xi \cdot x) \,,
\end{equation}
where $e(x) = e^{2\pi i x}$.
We rely on the following basic identities
\begin{equation}\label{F_Par}
    \sum_{x\in \Gr} |f(x)|^2
        =
            \frac{1}{N} \sum_{\xi \in \FF{\Gr}} \big|\widehat{f} (\xi)\big|^2 \,,
\end{equation}
and
\begin{equation}\label{f:inverse}
    f(x) = \frac{1}{N} \sum_{\xi \in \FF{\Gr}} \FF{f}(\xi) e(\xi \cdot x) \,.
\end{equation}
If
$$
    (f*g) (x) := \sum_{y\in \Gr} f(y) g(x-y) \quad \mbox{ and } \quad (f\circ g) (x) := \sum_{y\in \Gr} f(y) g(y+x)
$$
 then
\begin{equation}\label{f:F_svertka}
    \FF{f*g} = \FF{f} \FF{g} \quad \mbox{ and } \quad \FF{f \circ g} = \FF{f}^c \FF{g} = \ov{\FF{\ov{f}}} \FF{g} \,,
\end{equation}
where for a function $f:\Gr \to \mathbb{C}$ we put $f^c (x):= f(-x)$.
 Clearly,  $(f*g) (x) = (g*f) (x)$ and $(f\c g)(x) = (g \c f) (-x)$, $x\in \Gr$.
 The $k$--fold convolution, $k\in \N$
 is defined  by $*_k$,
 so $*_k := *(*_{k-1})$.

We use in the paper  the same letter to denote a set
$S\subseteq \Gr$ and its characteristic function $S:\Gr\rightarrow \{0,1\}.$
Write $\E(A,B)$ for the {\it additive energy} of two sets $A,B \subseteq \Gr$
(see e.g. \cite{tv}), that is
$$
    \E(A,B) = |\{ a_1+b_1 = a_2+b_2 ~:~ a_1,a_2 \in A,\, b_1,b_2 \in B \}| \,.
$$
If $A=B$ we simply write $\E(A)$ instead of $\E(A,A).$
Clearly,
\begin{equation}\label{f:energy_convolution}
    \E(A,B) = \sum_x (A*B) (x)^2 = \sum_x (A \circ B) (x)^2 = \sum_x (A \circ A) (x) (B \circ B) (x)
    \,.
\end{equation}

Put for any $A\subseteq \Gr$
$$
   \T_k (A) := | \{ a_1 + \dots + a_k = a'_1 + \dots + a'_k  ~:~ a_1, \dots, a_k, a'_1,\dots,a'_k \in A \} | \,.
$$

Let
\begin{equation}\label{f:E_k_preliminalies}
    \E_k(A)=\sum_{x\in \Gr} (A\c A)(x)^k \,,
\end{equation}
and
\begin{equation}\label{f:E_k_preliminalies_B}
\E_k(A,B)=\sum_{x\in \Gr} (A\c A)(x) (B\c B)(x)^{k-1}
    =\E(\Delta_k (A),B^{k}) \,,
\end{equation}
be the higher energies of $A$ and $B$.
Here
$$
    \Delta (A) = \Delta_k (A) := \{ (a,a, \dots, a)\in A^k \}\,.
$$
Similarly, we write $\E_k(f,g)$ for any complex functions $f$ and $g$.
Quantities $\E_k (A,B)$ can be
expressed
in terms of generalized convolutions (see \cite{ss_E_k}).


\begin{definition}
   Let $k\ge 2$ be a positive number, and $f_0,\dots,f_{k-1} : \Gr \to \C$ be functions.
Write $F$ for the vector $(f_0,\dots,f_{k-1})$ and $x$ for vector $(x_1,\dots,x_{k-1})$.
Denote by
${\mathcal C}_k (f_0,\dots,f_{k-1}) (x_1,\dots, x_{k-1})$
the function
$$
    \Cf_k (F) (x) =  \Cf_k (f_0,\dots,f_{k-1}) (x_1,\dots, x_{k-1}) = \sum_z f_0 (z) f_1 (z+x_1) \dots f_{k-1} (z+x_{k-1}) \,.
$$
Thus, $\Cf_2 (f_1,f_2) (x) = (f_1 \circ f_2) (x)$.
If $f_1=\dots=f_k=f$ then write
$\Cf_k (f) (x_1,\dots, x_{k-1})$ for $\Cf_k (f_1,\dots,f_{k}) (x_1,\dots, x_{k-1})$.
\end{definition}

\bigskip

Let $g : \Gr \to \C$ be a function, and $A\subseteq \Gr$ be a finite set.
By $\oT^{g}_A$ denote the matrix with indices in the set $A$
\begin{equation}\label{def:operator1}
    \oT^{g}_A (x,y) = g(x-y) A(x) A(y) \,.
\end{equation}
It is easy to see that $\oT^{g}_A$ is hermitian iff $\ov{g(-x)} = g(x)$.
The corresponding action of $\oT_A^g$ is
$$
    \langle \oT^{g}_A a, b \rangle = \sum_z g(z) (\ov{b} \c a) (z) \,.
$$
for any functions $a,b : A \to \C$.
In the case $\ov{g(-x)} = g(x)$ by $\Spec (\oT^{g}_A)$ we denote the spectrum of the operator $\oT^{g}_A$
$$
    \Spec (\oT^{g}_A) = \{ \mu_1 \ge \mu_2 \ge \dots \ge \mu_{|A|} \} \,.
$$
Write  $\{ f \}_{\a}$, $\a\in [|A|]$ for
the corresponding eigenfunctions.
General theory of such operators was developed in \cite{s,s_mixed}.

We conclude with few comments regarding the notation used in this paper.
For a positive integer $n,$ we set $[n]=\{1,\ldots,n\}.$
All logarithms are of base $2.$ Signs $\ll$ and $\gg$ are the usual Vinogradov symbols.
By $\d_0 (x)$ denote the delta--function, that is $\d_0 (0)=1$ and $\d_0 (x) = 0$ otherwise.

\section{Preliminaries}
\label{sec:preliminaries}

In the section $\Gr = \F_p$, where $p$ is a prime number
and $\Gamma\subseteq \F^*_p$ is a multiplicative subgroup.
A set $Q \subseteq \F^*_p$ is called {\it $\Gamma$--invariant} if $Q\G=Q$.
In the situation the following lemma which is a consequence of Stepanov's approach \cite{Stepanov}
can be formulated (see, e.g. \cite{K_Tula} or \cite{sv}).

\begin{lemma}
{
    Let $p$ be a prime number,
    $\Gamma\subseteq \F^*_p$ be a multiplicative subgroup, and
    $Q,Q_1\subseteq \F^*_p$ be any $\Gamma$--invariant sets such that
    $|Q||Q_1| \ll |\Gamma|^{4}$, $|Q| |Q_1| |\G|^2 \ll p^3$.
    Then
        \begin{equation}\label{f:improved_Konyagin_old3}
            \sum_{x\in \G} (Q \circ Q_1) (x) \ll |\Gamma|^{1/3} |Q|^{2/3} |Q_1|^{2/3} \,.
        \end{equation}
}
\label{l:improved_Konyagin_old}
\end{lemma}

Using Lemma \ref{l:improved_Konyagin_old}, one can easily deduce
upper bounds for moments of convolution of $\Gamma$ (see, e.g. \cite{ss}).

\begin{corollary}
    Let $p$ be a prime number and $\G \subseteq \F_p^*$ be a multiplicative subgroup, $|\G| \le p^{2/3}$.
    Then
    \begin{equation}\label{f:E_2_E_3}
        \E(\Gamma) \ll |\Gamma|^{5/2} \,, \quad \E_3 (\G) \ll |\G|^3 \log |\G| \,,
    \end{equation}
    and for all $l\ge 4$ the following holds
    \begin{equation}\label{f:E_l}
        \E_l (\G) = |\G|^l + O(|\G|^{\frac{2l+3}{3}}) \,.
    \end{equation}
\label{cor:E_l}
\end{corollary}

Certainly, the condition $|\G| \ll p^{2/3}$ in formula (\ref{f:E_l}) can be  relaxed.

The same method gives a generalization (see \cite{K_Tula}).

\begin{theorem}
    Let $\G \subseteq \F^*_p$ be a multiplicative subgroup, $|\G| < \sqrt{p}$.
    Let also $d\ge 2$ be a positive integer.
    Then arranging values of $(\G *_{d-1} \G) (\xi)$ in decreasing order
    $(\G *_{d-1} \G) (\xi_1) \ge (\G *_{d-1} \G) (\xi_2) \ge \dots $, where $\xi_j \neq 0$
    belong to distinct cosets,
    we have
    $$
        (\G *_{d-1} \G) (\xi_j) \ll_d |\G|^{d-2 + 3^{-1} (1+2^{2-d}) } j^{-\frac{1}{3}} \,.
    $$
    In particular
    \begin{equation}\label{f:T_d}
        \T_d (\G) \ll_d |\G|^{2d-2+2^{1-d}} \,.
    \end{equation}
\label{t:3_d_moment}
\end{theorem}

We need in a
lemma about Fourier coefficients of an arbitrary $\G$--invariant set (see e.g. \cite{K_Tula} and \cite{ss}).

\begin{lemma}
        Let $\G \subseteq \F^*_p$ be a multiplicative subgroup,
        and $Q$ be an $\G$--invariant subset of $\F^*_p$.
        Then for any $\xi \neq 0$ the following holds
\begin{equation}\label{f:G-inv_bound_F}
    | \FF{Q} (\xi) | \le \min \left\{ \left(\frac{|Q|p}{|\G|}\right)^{1/2} \,, \frac{|Q|^{3/4} p^{1/4} \E^{1/4} (\G)}{|\G|} \,,
                            p^{1/8} \E^{1/8} (\G) \E^{1/8} (Q) \left(\frac{|Q|}{|\G|}\right)^{1/2} \right\} \,.
\end{equation}
        Moreover for any positive integers $l$ and $m$ one has
\begin{equation}\label{f:M(G)_T_k}
        \oM (\G) \le p^{1/2lm} \T^{1/2lm}_l (\G) \T^{1/2lm}_m (\G) \cdot |\G|^{1-1/l-1/m} \,.
\end{equation}
\label{l:G-inv_bound_F}
\end{lemma}

In particular, the first formula of (\ref{f:G-inv_bound_F}) implies that $\oM (\G) \le \sqrt{p}$.
In the case of small ($|\G| > p^{\eps}$) multiplicative subgroups
nontrivial upper bounds for $\oM (\G)$ were obtained by additive combinatorics methods (see \cite{Bou_prod1,BGK,BK}).

\bigskip

We conclude the section recalling some results on the additive energy of multiplicative subgroups.
It was proved in \cite{H-K} (see also \cite{K_Tula})  that $\E(\G) = O(|\G|^{5/2})$,
provided by $|\G| \le p^{2/3}$.
At the moment the best upper bound for the additive energy of multiplicative subgroups
was obtained in \cite{s_ineq}.

\begin{theorem}
    Let $p$ be a prime number and $\G \subseteq \F^*_p$ be a multiplicative subgroup,
    $|\G| \le p^{2/3}$.
    Then
    \begin{equation}\label{f:subgroup_energy}
        \E (\G)
            \ll
                \min\{ |\G|^{\frac{32}{13}} \log^{\frac{41}{65}} |\G|,
                        |\G|^3 p^{-\frac{1}{3}} \log^{} |\G| + p^{\frac{1}{26}} |\G|^{\frac{31}{13}} \log^{\frac{8}{13}} |\G| \} \,.
    \end{equation}
\label{t:subgroup_energy}
\end{theorem}

\section{The proof of the main result}
\label{sec:proof}

We begin with a simple lemma.

\begin{lemma}
    Let $A\subseteq \Gr$ be a set and $g$ be a function, $\ov{g(-x)} = g(x)$.
    Let also $c$ be a complex constant, $g_c (x) := g(x) + c\d_0 (x)$.
    Then the operators $\oT^g_A$ and $\oT^{g_c}_A$ have the same eigenfunctions and
    $\Spec (\oT^{g_c}_A) = \Spec (\oT^{g}_A) + c$.
\label{l:operators_balanced}
\end{lemma}

In the proof we need a result on operators with special weights.
The method develops the arguments from \cite{ss_E_k,s_ineq,s_mixed}.

\begin{proposition}
    Let $A\subseteq \Gr$ be a set, and $\_phi \ge 0$, $\psi$ be two real functions of $\Gr$.
    Then
\begin{equation}\label{f:new_weights}
    \left( \sum_{x} \_phi (x) |\FF{f_1} (x)|^2 \right)^2
        \left( \frac{1}{|A|} \sum_{x} \psi (x) |\FF{A} (x)|^2 \right)^4
            \le
                \E_3 (A) \E_3 (\_phi, \psi, \psi) N^2 \,,
\end{equation}
    where $f_1$ is the main eigenfunction of the operator $\oT^{\FF{\psi}}_A$.
    In particular, if for some constant $c$ one has $\_phi (x) = \psi(x)+c$ then
\begin{equation}\label{f:new_weights'}
    \left( \frac{1}{|A|} \sum_{x} \_phi (x) |\FF{A} (x)|^2 \right)^2
        \left( \frac{1}{|A|} \sum_{x} \psi (x) |\FF{A} (x)|^2 \right)^4
            \le
                \E_3 (A) \E_3 (\_phi, \psi, \psi) N^2 \,.
\end{equation}
\label{p:new_weights}
\end{proposition}
\begin{proof}
    Consider operators $\oT_1 = \oT^{\FF{\_phi}}_A$, $\oT_2 = \oT^{\FF{\psi}}_A$.
    By assumption $\_phi,\psi$ are real functions,
    thus $\ov{\FF{\_phi} (-x)} = \FF{\_phi} (x)$, $\ov{\FF{\psi}(-x)} = \FF{\psi} (x)$
    and hence $\oT_1, \oT_2$ are hermitian matrices.
    We have
$$
    \sigma := \sum_{x,y,z} \oT_1 (x,y) \ov{\oT_2 (x,z)} \oT_2 (y,z)
        =
            \sum_{x,y,z\in A} \FF{\_phi} (x-y) \ov{\FF{\psi} (x-z)} \FF{\psi} (y-z)
                =
$$
$$
                =
                    \sum_{a,b} \Cf_3 (A) (a,b) \FF{\_phi} (a-b) \ov{\FF{\psi} (a)} \FF{\psi} (b) \,.
$$
Using the Cauchy--Schwartz inequality, the Fourier transform
and formula
$$
    \sum_{a,b} \Cf^2_3 (A) (a,b) = \E_3 (A) \,,
$$
see e.g. \cite{sv}, we have
$$
    \sigma^2
        \le
            \E_3 (A) \sum_{a,b} |\FF{\_phi} (a-b)|^2 |\FF{\psi} (a)|^2 |\FF{\psi} (b)|^2
                =
                    \E_3 (A) \E_3 (\_phi, \psi, \psi) N^2 \,.
$$
On the other hand, let $\{ \mu_\a \}$, $\a \in [|A|]$ be the spectrum of the operator $\oT_2$ and
let $\{ f_\a \}$, $\a \in [|A|]$ be the correspondent eigenfunctions.
Substituting  the formula
$$
    \oT_2 (x,y) = \sum_{\a} \mu_\a f_\a (x) \ov{f_\a (y)}
$$
into the definition of the quantity $\sigma$, we get
$$
    \sigma = \sum_{\a} |\mu_\a|^2 \cdot \langle \oT_1 f_\a, f_\a \rangle \,.
$$
By assumption $\_phi (x) \ge 0$ and hence $\oT_1$ is a nonnegatively defined operator.
Thus, we have  (\ref{f:new_weights}).
To obtain (\ref{f:new_weights'}) note that $\FF{\_phi} = \FF{\psi} + c N\d_0$
and by Lemma \ref{l:operators_balanced} we get  in the situation
$$
    \sigma = \sum_{\a} |\mu_\a|^2 \mu_\a (\oT_1) \ge |\mu_1|^2 \mu_1 (\oT_1)
        \ge
    \left( \frac{1}{|A|} \sum_{x} \psi (x) |\FF{A} (x)|^2 \right)^2
            \cdot
        \left( \frac{1}{|A|} \sum_{x} \_phi (x) |\FF{A} (x)|^2 \right) \,.
$$
In the last inequality we have used the variational principle, see e.g. \cite{Horn-Johnson}.
This completes the proof.
$\hfill\Box$
\end{proof}

\bigskip

Clearly, inequality (\ref{f:new_weights'}) holds in the case of any three different nonnegative functions.


Now let us formulate the main result of the paper.


\begin{theorem}
    Let $\G \subseteq \F_p$ be a multiplicative subgroup, $|\G| \le p^{2/3}$.
    Then
\begin{equation}\label{f:M_2/3}
    \oM (\G) \ll |\G|^{1/2} p^{1/6} \log^{1/6} |\G| \,.
\end{equation}
    Further, if
\begin{equation}\label{cond:not_so_small}
    p^{10} \ll |\G|^{22} \E^{2} (\G) \cdot \log^{9} |\G|
\end{equation}
    then
\begin{equation}\label{f:13/2}
    \sum_{\xi \neq 0} |\FF{\G} (\xi)|^{32/5} \ll |\G|^{16/5} \E^{2/5} (\G) p^{6/5} \cdot \log^{9/5} |\G| \,.
\end{equation}
\label{t:M_2/3}
\end{theorem}
\begin{proof}
    Put $t = |\G|$, $\E = \E(\G)$.
    There is $\xi \neq 0$ such that
    \begin{equation}\label{tmp:20.11.2013}
        \oM^2 := \oM^2 (\G) = t^{-1} \sum_{x\in \xi \G} |\FF{\G} (x)|^2
            = \langle \oT^{\FF{\xi \G}}_\G \G(x)/ t^{-1/2}, \G(x)/ t^{-1/2} \rangle \,.
    \end{equation}
    Using Proposition \ref{p:new_weights} with $\_phi = \psi = \xi \G$, we obtain by Lemma \ref{cor:E_l}
    $$
        \oM^{12} \le \E^2_3 (\G) p^2 \ll t^6 p^2 \log^2 t
    $$
    and inequality (\ref{f:M_2/3}) follows.

    To get  (\ref{f:13/2}), for any $\rho \in (0,\oM]$ consider the set
    $$
        Q=Q_\rho = \{ x\in \F^*_p ~:~ \rho < |\FF{\G} (x)| \le 2\rho \} \,.
    $$
    Let $q=|Q|$ and let us obtain an upper bound for $q$.
    Without loss of generality we can assume that $q>0$.
    Take any element $\xi$ from $Q$.
    Using  Lemma \ref{cor:E_l} and Proposition \ref{p:new_weights} with $\_phi = \xi \G$, $\psi = Q$, we have
    \begin{equation}\label{f:E_3,E_3_int}
        \left( \frac{q}{t} \right)^4 \rho^{12}
            \le
                \E_3 (\G) p^2 \E_3 (Q,Q,\xi \G)
                    \ll
                        t^3 p^2 \log t \cdot \E_3 (Q,Q,\xi \G) \,.
    \end{equation}
    Our task is to estimate $\E_3 (Q,Q,\xi \G)$.
    Applying the pigeonhole principle, we find $\omega$
    and a set $S_\omega \subseteq ((Q-Q)\cap (\xi \G - \xi \G)) \setminus \{ 0 \}$
    with the property $(Q \c Q) (x)$
    differ at most twice on $S_\omega$ and such that
    $$
        \E_3 (Q,Q,\xi \G) \ll q^2 t  + \omega \E(Q,\xi \G) \cdot \log t  \,.
    $$
    Using Lemma \ref{l:improved_Konyagin_old} and
    assuming that $q \le t^2$, we obtain
    \begin{equation}\label{tmp:15.11.2013_1}
        \E_3 (Q,Q,\xi \G) \ll q^2 t  + \omega t q^{3/2} \cdot \log t \,,
    \end{equation}
    provided by $q^{3/2} t^3 \ll p^{3}$.
    Further, the number $\o$ is bounded by
    $$
        \o \le
                \max_{x\neq 0} (Q\c Q) (x) \le \rho^{-4} \max_{x\neq 0} \sum_y |\FF{\G} (y)|^2 |\FF{\G} (x+y)|^2
            \le
                \rho^{-4} \E p \,.
    $$
    Thus, substitution of the last two estimates into (\ref{tmp:15.11.2013_1}) gives us
    \begin{equation}\label{tmp:15.11.2013_2}
        \E_3 (Q,Q,\xi \G) \ll q^2 t + t q^{3/2} (\rho^{-4} \E p)^{} \cdot \log t \,.
    \end{equation}
    It is easy to see that the second term in (\ref{tmp:15.11.2013_2}) dominates.
    Indeed, by formula (\ref{f:T_d}) of Theorem \ref{t:3_d_moment}, we get
    $$
        q \rho^8 \le p \T_4 (\G) \ll p t^{4+1/8} \le \E^2 p^2 \cdot \log^2 t
    $$
    because of $\E \ge t^2$ and $t<p$.
    Thus
    $$
        \left( \frac{q}{t} \right)^4 \rho^{12}
            \ll
                t^4 p^2 q^{3/2} (\rho^{-4} \E p)^{} \cdot \log^2 t
    $$
    and after some calculations, we get
    \begin{equation}\label{tmp:18.11.2013_1}
        q \ll (t^{16} p^6 \E^2 )^{1/5} \rho^{-32/5} \cdot \log^{4/5} t \,.
    \end{equation}
    It is easy to check that  (\ref{tmp:18.11.2013_1}) implies (\ref{f:13/2}).
    Indeed, for any parameter $\D>0$ and the Parseval identity, we have
$$
    \sigma:= \sum_{\xi \neq 0} |\FF{\G} (\xi)|^{32/5}
        \ll
            \D^{22/5} tp
                +
            t^{16/5} \E^{2/5} (\G) p^{6/5} \cdot \log^{9/5} t
                =
            \D^{22/5} tp + s \,,
$$
provided by $|Q_\D| \le t^2$ and $|Q_\D|^{3/2} t^3 \ll p^3$.
Taking $\D= (s/(tp))^{5/22}$, we see
$\sigma \ll s$.
It remains to check  that $|Q_\D| \le t^2$ and $|Q_\D|^{3/2} t^3 \ll p^3$.
By (\ref{tmp:18.11.2013_1}), we have
$$
    |Q_\D| \ll s \D^{-32/5} = (tp)^{16/11} s^{-5/11}
$$
and, hence, the first inequality is equivalent to
\begin{equation}\label{tmp:19.11.2013_1}
    p^{10} \ll t^{22} \E^{2} \cdot \log^{9} t \,,
\end{equation}
which is our condition.
Similarly, the second estimate is equivalent to
\begin{equation}\label{tmp:19.11.2013_2}
    t^{66} \ll p^{36} \E^{6} \cdot \log^{27} t \,.
\end{equation}
Inequality (\ref{tmp:19.11.2013_2}) follows easily from a trivial bound $\E \ge t^2$ and the assumption $t \le p^{2/3}$.
This completes the proof.
$\hfill\Box$
\end{proof}

\bigskip

In view of Theorem \ref{t:subgroup_energy} and formulas (\ref{f:G-inv_bound_F}), (\ref{f:M(G)_T_k})
of Lemma \ref{l:G-inv_bound_F} our bound (\ref{f:M_2/3}) beats (up to logarithms)
the previous estimates of $\oM (\G)$
in the interval
$|\G| \in (p^{52/141}, p^{29/48})$,
roughly,
and coincide with it in the  interval $|\G| \in (p^{29/48}, p^{2/3})$.
The constant $52/141$ appears if one take $l=3$, $m=2$ in formula (\ref{f:M(G)_T_k})
of Lemma \ref{l:G-inv_bound_F} and apply first bound of Theorem \ref{t:subgroup_energy}
as well as Theorem \ref{t:3_d_moment}.

\begin{remark}
Condition (\ref{cond:not_so_small}) says that the size of our subgroup $\G$ is not so small.
Using  a trivial bound $\E (\G) \ge |\G|^2$, we have that (\ref{cond:not_so_small}) takes place
for any multiplicative subgroup $\G$ such that $p^{5/13} \log^{-9/26} |\G| \ll |\G| \le p^{2/3}$.
\end{remark}

\begin{remark}
    The sum $\sigma$ from (\ref{f:13/2}) can be estimated as
$$
    \sigma \le p \T_3 (\G) \oM^{2/5} (\G) \,.
$$
    In view of inequality (\ref{f:T_d}) of Theorem \ref{t:3_d_moment}
    as well as bound (\ref{f:M_2/3}), we have
\begin{equation}\label{tmp:18.11.2013_3}
    \sigma \ll p |\G|^{17/4} (|\G|^{1/2} p^{1/6} \log^{1/6} |\G|)^{2/5} \,.
\end{equation}
    One can check that our estimate (\ref{f:13/2}) is better than (\ref{tmp:18.11.2013_3})
    in the range
    $|\G| \in ( p^{116/231}, p^{2/3} )$
    (up to logarithms).
\end{remark}

\begin{remark}
    The arguments of the proof of inequality (\ref{f:M_2/3}) show also that
\begin{equation}\label{f:char_6}
    \max_{\xi \neq 0} \left| \sum_{x} \chi(x) e^{\frac{2 \pi i x \xi}{p}} \right|
            \ll
        |\G|^{1/2} p^{1/6} \log^{1/6} |\G| \,,
\end{equation}
where $\chi$ is any multiplicative character on a subgroup $\G$, $|\G| \le p^{2/3}$.
Such sums are studied in \cite{Shp-DH}.
Moreover, using the variational principle (see e.g. \cite{Horn-Johnson})
and the convexity of the function $z\to z^3$ or just Proposition 45 of paper \cite{ss_E_k},
we get for any $x\in \F_p^*$
\begin{equation}\label{f:char_6'}
    \sum_{\a=1}^{|\G|} |\FF{\chi}_\a (x)|^6 \le p \E_3 (\G) \,,
\end{equation}
where $\{ \chi_\a \}_{\a\in [|\G|]}$ forms the orthogonal  family of all  multiplicative characters on $\G$.
\end{remark}

An application of Theorem \ref{t:M_2/3} gives us a result on basis properties
of multiplicative subgroups.

\begin{corollary}
    Let $\Gamma \subseteq \F_p^*$ be a multiplicative subgroup such that $|\Gamma| \gg p^{1/2} \log^{1/3} p$,
    and $-1 \in \G$.
    Then for all sufficiently large $p$, we have $\F^*_p \subseteq 5\Gamma$.
\label{c:6R_subgroups}
\end{corollary}
\begin{proof}
Put $S = \Gamma+\Gamma = \G - \G$, $n=|\Gamma|$, $m = |S|$, and $\rho = \oM (\G)$.
Applying inequality (\ref{f:M_2/3}) of Theorem \ref{t:M_2/3}, we get $\rho \le n^{1/2} p^{1/6} \log^{1/6} n$.
If $\F_p^* \not\subseteq 5\Gamma$ then for some $\lambda\neq 0,$ we
obtain
$$
    0 = \sum_{\xi} \FF{S}^2 (\xi) \FF{\Gamma} (\xi) \FF{\lambda \Gamma} (\xi)
        =
            m^2n^2+\sum_{\xi\neq 0} \FF{S}^2 (\xi) \FF{\Gamma} (\xi) \FF{\lambda \Gamma} (\xi) \,.
$$
Therefore, by the upper bound for $\rho$ and the Parseval identity, we get
$$
    n^2m^2
        \le \rho^2 mp
        \ll (n^{1/2} p^{1/6} \log^{1/6} n)^2 mp \,.
$$
Now by Theorem 1.1 from \cite{ss}, say, that is
$m \gg \min\{ n p^{1/3} \log^{-1/3} n, n^{7/3} p^{-1/3} \log^{-2/3} n \}$,
and the assumption $|\Gamma| \gg p^{1/2} \log^{1/3} p$, we obtain the required result.
$\hfill\Box$
\end{proof}


\section{An application to Heilbronn's exponential sum}
\label{sec:application}

Heilbronn's exponential sum is defined by
\begin{equation}\label{def:Heilbronn_sum}
    S(a) = \sum_{n=1}^p e^{2 \pi i \cdot \frac{an^p}{p^2} } \,.
\end{equation}

D.R. Heath--Brown obtained in \cite{H} the first nontrivial upper bound for the sum.
After that the result was improved in papers \cite{H-K}, \cite{s_heilbronn}, \cite{sv_heilbronn}.

The method of the previous section gives a new upper bound for Heilbronn's exponential sum.

\begin{theorem}
    Let $p$ be a prime, and $a\neq 0 \pmod p$.
    Then
    $$
        |S(a)| \ll p^{\frac{5}{6}} \log^{\frac{1}{6}} p \,.
    $$
\label{t:main-}
\end{theorem}

Indeed, consider
the following multiplicative subgroup
\begin{equation}\label{def:H_Gamma}
    \G = \{ m^p ~:~ 1\le m \le p-1 \} = \{ m^p ~:~ m \in \Z/(p^2 \Z) \,, m \neq 0 \} \subseteq \Z/(p^2 \Z) 
\end{equation}
and note that $\max_{a\neq 0} |S(a)| = \oM (\G)$.

\bigskip

To obtain Theorem \ref{t:main-},
we need in a lemma, see e.g. \cite{s_heilbronn},
which is another consequence of Stepanov's method \cite{H-K}.

\begin{lemma}
    For Heilbronn's subgroup (\ref{def:H_Gamma}), one has
\begin{equation}\label{f:E_3_Heilbronn}
    \E_3 (\G) \ll p^3 \log p \,.
\end{equation}
\end{lemma}

After that apply the arguments of the proof of Theorem \ref{t:M_2/3}, Fourier transform on $\Z/(p^2\Z)$
 and use bound (\ref{f:E_3_Heilbronn}).
It gives us Theorem \ref{t:main-}.

\section{Appendix}
\label{sec:appendix}

Corollary \ref{cor:E_l} holds for subgroups of size $O(p^{2/3})$.
Now we extend the result of the statement on $\E_3 (\G)$ for large subgroups.

\begin{proposition}
    Let $\G \subseteq \F_p$ be a multiplicative subgroup, $p^{1/2} \ll |\G| \ll p^{3/4}$.
    Then
\begin{equation}\label{f:E_3_large}
    \E_3 (\G) = \frac{|\G|^6}{p^2} + O( |\G|^3 \log^3 |\G|) + O(p^{1/3} \oM^{4/3} (\G) |\G|^{5/3} \log^3 |\G|) \,.
\end{equation}
    Moreover, for any $\G$--invariant set $S$, we get
$$
    \E_3 (\G,S) = \frac{|\G|^2 |S|^4}{p^2} + O\left( |S|^2 |\G| \log^3 |\G| + \frac{|S|^3 \oM^2 (\G)}{p} \right)
                        +
$$
\begin{equation}\label{f:E_3_large_S}
                        + O(p^{2/3} \oM^{2/3} (\G) |S|^2 |\G|^{-1/3} \log^3 |\G|) \,,
\end{equation}
and
\begin{equation}\label{f:E_3_large_S'}
    \E_3 (S) = \frac{|S|^6}{p^2}
                    + O(|S|^4 |\G|^{-1})
                        + O(p^{}  |S|^3 |\G|^{-4/3} \log^3 |\G|) \,.
\end{equation}
\label{p:E_3_large}
\end{proposition}
\begin{proof}
Let $t=|\G|$, $\oM = \oM(\G)$.
By the first bound of Lemma \ref{l:G-inv_bound_F} and formulas (\ref{F_Par}), (\ref{f:F_svertka}), we have
$$
    \E_3 (\G) = \frac{1}{p^2} \sum_{x,y} |\FF{\G} (x)|^2 |\FF{\G} (y)|^2 |\FF{\G} (x-y)|^2
        =
$$
\begin{equation}\label{tmp:16.11.2013_1}
        =
            \frac{t^6}{p^2} + O\left( \frac{t^3 \oM^2}{p} \right) + \frac{1}{p^2} \sum_{x \neq 0,\, y \neq 0,\, x-y \neq 0} |\FF{\G} (x)|^2 |\FF{\G} (y)|^2 |\FF{\G} (x-y)|^2
                =
                    \frac{t^6}{p^2} + O\left( \frac{t^3 \oM^2}{p} \right) + \sigma \,.
\end{equation}
Let $\omega = t^{1/2}$, and
$$
    \Omega_j = \{ \xi \in \F^*_p ~:~ \omega 2^{j-1} < |\FF{R} (\xi)| \le \omega 2^{j} \}\,,
                        \quad j \in [l],\, \quad l \ll \log t \,.
$$
Clearly, there are $j_1,j_2,j_3 \in [l]$ such that
$$
    \sigma \le \frac{l^3}{p^2} \sum_{x\in \Omega_{j_1},\, y\in \Omega_{j_2},\, x-y \in \Omega_{j_3}} |\FF{\G} (x)|^2 |\FF{\G} (y)|^2 |\FF{\G} (x-y)|^2 \,.
$$
Let $Q_i = \Omega_{j_i}$, $q_i = |Q_i|$, and $\rho_i = \omega 2^{j_i}$, where $i\in [3]$.
Then
\begin{equation}\label{f:error_term_for_E_3}
    \sigma
        \le
            \frac{l^3}{p^2} \rho^2_1 \rho^2_2 \rho^2_3 \sum_{x\in Q_1} (Q_2 * Q_3) (x) \,.
\end{equation}
Suppose that $q_1 q_2 q_3 \gg p^3 t^{-1}$.
Using the first bound of Lemma \ref{l:G-inv_bound_F} and Parseval identity (\ref{F_Par}), we get
$$
    \sum_{x\in Q_1} (Q_2 * Q_3) (x) = \frac{1}{p} \sum_z \FF{Q}_1 (z) \FF{Q}_2 (z) \FF{Q}_3 (z)
        =
            O \l( \frac{q_1 q_2 q_3}{p} \r) \,.
$$
Whence, by a trivial bound $q_j \ll pt \rho^{-2}_j$, we obtain
$$
    \sigma \ll \frac{l^3}{p^3} \rho^2_1 \rho^2_2 \rho^2_3 q_1 q_2 q_3 \ll l^3 t^3
$$
and the result follows  in the case.

Now suppose that $q_1 q_2 q_3 \ll p^3 t^{-1}$.
Clearly, $q_j \ge t$, $j=1,2,3$.
If $t q_2 q_3 \gg t^{5}$ then we have a contradiction with the assumption  $t\gg p^{1/2}$.
Thus, $t q_2 q_3 \ll t^{5}$, and, similarly, $t q_1 q_2 \ll t^{5}$ and $t q_1 q_3 \ll t^{5}$.
Because of  $q_1 q_2 q_3 \ll p^3 t^{-1}$  and
$q_1 \ge t$,
we get
$t^2 q_2 q_3 \ll p^3$.
Returning to formula (\ref{f:error_term_for_E_3}) and
applying
Lemma \ref{f:improved_Konyagin_old3}, we obtain
$$
    \sigma
        \ll
            \frac{l^3}{p^2} \rho^2_1 \rho^2_2 \rho^2_3 \cdot \frac{q_1 (q_2 q_3)^{2/3}}{t^{2/3}} \,.
$$
Using the bound $q_j \ll pt \rho^{-2}_j$ several times
as well as the estimate $\rho_j \le \oM$, we get
$$
    \sigma \ll \frac{l^3}{t^{2/3} p^2} \rho^{2/3}_2 \rho^{2/3}_3 (pt)^{7/3}
        \ll
            l^3 t^{5/3} \oM^{4/3} (\G) p^{1/3}
$$
as required.
The proofs of formulas (\ref{f:E_3_large_S}), (\ref{f:E_3_large_S'}) are similar,
one just need to apply  Lemma \ref{l:G-inv_bound_F}
for $\G$--invariant set $S$ not $\G$.
$\hfill\Box$
\end{proof}

\begin{remark}
The term $O(p^{1/3} \oM^{4/3} (\G) |\G|^{5/3} \log^3 |\G|)$ in (\ref{f:E_3_large}) is quite tight up to
our current knowledge of multiplicative subgroups.
Indeed, suppose that we have $\oM(\G) \sim p$ and $|\FF{\G}(x)|^2$, $|\FF{\G}(y)|^2$,
$|\FF{\G}(x-y)|^2 \sim \oM^2 (\G)$
in formula (\ref{tmp:16.11.2013_1}).
The choice of the Fourier coefficients is admissible  in view of the first bound from formula (\ref{f:G-inv_bound_F})
of Lemma \ref{l:G-inv_bound_F}.
Now suppose that for $x,y\in \xi \G$ the following holds $(\xi G \c \xi \G) (x-y) \gg |\G|^{2/3}$.
This choice is also possible in view of Lemma \ref{l:improved_Konyagin_old}.
Thus, in the situation we get $p |\G|^{5/3}$ which coincide with  (\ref{f:E_3_large}).
\end{remark}

\begin{remark}
For $|\G| \gg p^{3/4}$ another
asymptotic formulas hold.
\end{remark}

Using formulas (\ref{f:E_3_large}), (\ref{f:E_3_large_S}), (\ref{f:E_3_large_S'})
as well as an "asymmetric"\, (applied to the correspondent balanced functions)
case of Proposition \ref{p:new_weights} one can obtain
upper bounds for sums similar to (\ref{f:13/2})
in the situation when  the  size of our subgroups
greater than $p^{2/3}$.
We do not make such calculations here.

\bigskip

\no{Division of Algebra and Number Theory,\\ Steklov Mathematical
Institute,\\
ul. Gubkina, 8, Moscow, Russia, 119991}
\\
and
\\
Delone Laboratory of Discrete and Computational Geometry,\\
Yaroslavl State University,\\
Sovetskaya str. 14, Yaroslavl, Russia, 150000
\\
and
\\
IITP RAS,  \\
Bolshoy Karetny per. 19, Moscow, Russia, 127994\\
{\tt ilya.shkredov@gmail.com}

\end{document}